\documentclass[11pt]{article}

\usepackage{amsmath,amssymb,amsthm}
\usepackage[colorlinks=true,citecolor=black,linkcolor=black,urlcolor=blue]{hyperref}

\usepackage{cite}
\usepackage{cleveref}
\crefname{equation}{}{}
\Crefname{equation}{Equation}{Equations}
\crefname{theorem}{Theorem}{Theorems}
\Crefname{theorem}{Theorem}{Theorems}
\crefname{lemma}{Lemma}{Lemmas}
\Crefname{lemma}{Lemma}{Lemmas}
\crefname{proposition}{Proposition}{Propositions}
\Crefname{proposition}{Proposition}{Propositions}
\crefname{corollary}{Corollary}{Corollaries}
\Crefname{corollary}{Corollary}{Corollaries}
\crefname{conjecture}{Conjecture}{Conjectures}
\Crefname{conjecture}{Conjecture}{Conjectures}
\crefname{section}{Section}{Sections}
\Crefname{section}{Section}{Sections}
\crefname{example}{Example}{Examples}
\Crefname{example}{Example}{Examples}
\crefname{problem}{Problem}{Problems}
\Crefname{problem}{Problem}{Problems}
\crefname{remark}{Remark}{Remarks}
\Crefname{remark}{Remark}{Remarks}
\crefname{figure}{Figure}{Figures}
\Crefname{figure}{Figure}{Figures}

\newtheorem{theorem}{Theorem}
\newtheorem{lemma}[theorem]{Lemma}

\newcommand{\doi}[1]{\href{http://dx.doi.org/#1}{\texttt{doi:#1}}}
\newcommand{\arxiv}[1]{\href{http://arxiv.org/abs/#1}{\texttt{arXiv:#1}}}

\title{Weakly norming graphs are edge-transitive}

\author{Alexander Sidorenko \\
\small\tt sidorenko.ny@gmail.com}

\date{\today}

\begin{document}
\maketitle

\begin{abstract}
Let $\mathcal{H}$ be the class of 
bounded measurable symmetric functions on $[0,1]^2$.
For a function $h \in \mathcal{H}$
and a graph $G$
with vertex set $\{v_1,\ldots,v_n\}$  
and edge set $E(G)$,
define
\[
  t_G(h) \; = \; 
  \int \cdots \int \prod_{\{v_i,v_j\} \in E(G)} h(x_i,x_j)  
    \: dx_1 \cdots dx_n \: .
\]
Answering a question raised by Conlon and Lee, 
we prove that in order 
for $t_G(|h|)^{1/|E(G)|}$ to be a norm on $\mathcal{H}$, 
the graph $G$ must be edge-transitive.
\end{abstract}

\vspace{0.4cm}

\noindent
Let $\mathcal{H}$ be the class of 
bounded measurable symmetric functions on $[0,1]^2$, 
and $\mathcal{H}_+$ be the subclass 
of nonnegative functions in $\mathcal{H}$. 
The functions from $\mathcal{H}_+$ with values in $[0,1]$ 
are known as ``graphons'' (see \cite{Lovasz:2010}). 
For a function $h \in \mathcal{H}$
and a graph $G$
with vertex set $\{v_1,\ldots,v_n\}$  
and edge set $E(G)$,
the \emph{homomorphism density} from $G$ to $h$ is defined as
\[
  t_G(h) \; = \; 
  \int \cdots \int \prod_{\{v_i,v_j\} \in E(G)} h(x_i,x_j)  
    \: dx_1 \cdots dx_n \: .
\]

Hatami \cite{Hatami:2010} and Lov\'{a}sz \cite[Chapter 14.1]{Lovasz:2010}
posed the problem of determining 
when $t_G(h)$ provides a (semi-)norm on $\mathcal{H}$. 
A graph $G$ is called \emph{norming} 
if $|t_G(h)|^{1/|E(G)|}$ is a semi-norm,
and \emph{weakly norming} if $t_G(|h|)^{1/|E(G)|}$ is a norm. 
Every norming graph is weakly norming. 
Lee and Sch\"{u}lke~\cite{Lee:2019} demonstrated that 
$G$ is (weakly) norming if and only if 
the functional $t_G(\cdot)$ is strictly convex on (nonnegative) functions 
$h \in \mathcal{H}$.
Conlon and Lee \cite[Section~6]{Conlon:2017} proposed some conjectures 
characterizing norming and weakly norming graphs. 

Every weakly norming graph is necessarily bipartite 
(see~\cite{Hatami:2010,Lovasz:2010}). 
Thus, it is possible to define norming and weakly norming graphs 
using asymmetric functions $h$ 
(as was done by Hatami in~\cite{Hatami:2010}). 
It is easy to see that every (weakly) norming graph 
in the asymmetric setting 
is (weakly) norming in the symmetric setting as well. 
For this reason, we will restrict our attention to the symmetric case.

All known examples of norming and weakly norming graphs 
are edge-transitive 
(see \cite{
Conlon:2017,Garbe:2019,Hatami:2010,Kral:2019,Lee:2019,Lovasz:2010}). 
Thus, it is natural to ask 
if every edge-transitive bipartite graph is weakly norming, 
and if every weakly norming graph is edge-transitive 
(see \cite[Section 6]{Conlon:2017}). 
The first question was recently answered in the negative 
by Kr\'{a}l’ et al \cite{Kral:2019} who proved that 
toroidal grids $C_{2k} \Box C_{2k}$ with $k \geq 3$ 
are not weakly norming. 
In this note, we provide the positive answer to the second question. 

\begin{theorem}\label{th:transitive}
Every weakly norming graph is edge-transitive.
\end{theorem}

To prove \cref{th:transitive}, we need a few results from 
\cite{Garbe:2019}, \cite{Hatami:2010}, and \cite{Lovasz:2010}. 

\begin{theorem}[{\cite[Theorem 1.2]{Garbe:2019}}]\label{th:connected}
A graph is weakly norming if and only if 
all its non-singleton connected components 
are isomorphic and weakly norming.
\end{theorem}

Let $E(G)=\{e_1,e_2,\ldots,e_k\}$, $e_l = \{v_{i(l)},v_{j(l)}\}$.  
Let $h_1,h_2,\ldots,h_k \in \mathcal{H}$. 
Define
\[
  t_G(h_1,h_2,\dots,h_k) \; = \; 
  \int \cdots \int \prod_{l=1}^k h_l(x_{i(l)},x_{j(l)})  
    \: dx_1 \cdots dx_n\: .
\]
We clearly have $t_G(h,h,\ldots,h) = t_G(h)$. 

\begin{theorem}[{\cite[Theorem 14.1]{Lovasz:2010}}]\label{th:Hatami}
$G$ is weakly norming if and only if 
for any $h_1,h_2,\ldots,h_k \in \mathcal{H}_+$,
\[
  t_G(h_1,h_2,\dots,h_k)^k \; \leq \;
  \prod_{l=1}^k t_G(h_l) \: .
\]
\end{theorem}

The following result was proved by Hatami in the asymmetric setting, 
but the proof remains valid in the symmetric setting 
(see \cite[Remark~2.11]{Hatami:2010}).

\begin{theorem}[{\cite[Theorem 2.10 (ii)]{Hatami:2010}}]\label{th:degree}
If a connected bipartite graph $G$ is weakly norming, 
then all vertices belonging to the same part have equal degrees.
\end{theorem}

\begin{theorem}[{\cite[Theorem 5.29]{Lovasz:2010}}]\label{th:Lovasz}
If $t_F(h)=t_G(h)$ for all 
$h \in \mathcal{H}_+$, then $F$ and $G$ are isomorphic.
\end{theorem}

\begin{lemma}\label{th:lemma}
If a graph $G$ with edges $e_1,e_2,\ldots,e_k$ 
is weakly norming, then for any $h \in \mathcal{H}_+$,
\[
  t_{G-e_1}(h) \: = \:
  t_{G-e_2}(h) \: = \: \cdots \: = \:
  t_{G-e_k}(h) \: .
\]
\end{lemma}

\begin{proof}[\bf{Proof}]
We say that a function $h \in \mathcal{H}_+$ is 
\emph{separated from zero} 
if there exists $\delta_h > 0$ such that 
$h(x,y) \geq \delta_h$ for all $x,y \in [0,1]$. 
Since any nonnegative function in $\mathcal{H}_+$ 
is a limit (in $\mathcal{L}_1$ metric) 
of a sequence of functions separated from zero, 
and the functional $t_G(\cdot)$ is continuous, 
it is sufficient to prove that 
$t_{G-e_1}(h) = t_{G-e_2}(h)$ for any $h$ that is separated from zero. 
Let ${\mathbf 1}$ denote the constant $1$ function,
and let $h+\varepsilon$ be a shorthand for 
$h+\varepsilon \cdot {\mathbf 1}$. 
If $|\varepsilon| \leq \delta_h$, 
then both $h+\varepsilon$ and  $h-\varepsilon$ are nonnegative. 
Denote 
\[
  t_l \; = \; t_{G-e_l}(h) \; = \; 
  t_G(h,\ldots,h,\: h_l={\mathbf 1},\: h,\ldots,h) \: .
\]
Notice that 
\[
  t_G(h+\varepsilon,h-\varepsilon,h,\ldots,h) \; = \;
  t_G(h) + (t_1 - t_2) \varepsilon + O(\varepsilon^2) \: ,
\]
\[
  t_G(h+\varepsilon) \; = \;
  t_G(h) + (t_1+t_2+\ldots+t_k) \varepsilon + O(\varepsilon^2) \: .
\]
By \cref{th:Hatami},
\[
  t_G(h+\varepsilon,h-\varepsilon,h,\ldots,h)^k \; \leq \;
  t_G(h+\varepsilon) \: t_G(h-\varepsilon) \: t_G(h)^{k-2} \: .
\]
Since 
$t_G(h+\varepsilon) \: t_G(h-\varepsilon) = t_G(h)^2 + O(\varepsilon^2)$,
we get
\[
  t_G(h)^k + k t_G(h)^{k-1}(t_1 - t_2) \varepsilon + O(\varepsilon^2)
    \; \leq \;
  t_G(h)^k + O(\varepsilon^2) \: .
\]
As $\varepsilon$ can be of any sign, and $t_G(h) \neq 0$, 
we conclude that $t_1-t_2 = 0$. 
\end{proof}

\begin{proof}[\bf{Proof of \cref{th:transitive}}]
By \cref{th:connected}, we may assume that $G$ is connected.
As $G$ is connected and bipartite, 
the two parts of its vertex set are well defined, 
we denote them $A$ and $B$.
By \cref{th:degree}, there are integers $a,b$ such that 
all vertices from $A$ have degree $a$, 
and all vertices from $B$ have degree $b$. 
Without loss of generality, we may assume that $a \leq b$.
If $a=1$, then $G$ is a star, and is edge-transitive.
We will assume that $a\geq 2$.
Let $\{e_1,e_2,\ldots,e_k\}$ be the edge set of $G$.
Let $G_l$ denote the subgraph of $G$ obtained by dropping edge $e_l$. 
Since $G$ is connected, and each vertex has degree at least $2$, 
then $G$ has an edge which is not a bridge, 
so at least one of the subgraphs $G_1,G_2,\ldots,G_k$ must be connected. 
By \cref{th:lemma,th:Lovasz}, the graphs $G_1,G_2,\ldots,G_k$ are isomorphic, 
hence, each of them is connected.

We are going to demonstrate that there is an automorphism of $G$ 
which maps edge $e_1$ into $e_2$. 
As $G_1$ and $G_2$ are isomorphic, 
there exists a permutation $\pi$ of the vertices of $G$ 
which provides such an isomorphism. 
It means that if $l \geq 2$, then 
$\pi$ maps $e_l$ into one of the edges among $e_1,e_3,\ldots,e_k$. 
We claim that $\pi$ maps $e_1$ into $e_2$.
Indeed, let $e_1 = \{u_1,v_1\}$, $e_2=\{u_2,v_2\}$ where 
$u_1,u_2 \in A$ and $v_1,v_2 \in B$. 
It is possible that either $u_1=u_2$ or $v_1=v_2$, 
but not simultaneously. 
The degree of $u_1$ in $G_1$ is $a-1$,
while the degree of any vertex in $G_2$, apart from $u_2$ and $v_2$, 
is either $a$ or $b$. 
Recall that $a \leq b$.
Thus, either $\pi(u_1)=u_2$, or $a=b$ and $\pi(u_1)=v_2$. 
In the second case, 
the only vertex of degree $a-1=b-1$ in $G_2$ (apart from $v_2$) is $u_2$, 
hence $\pi(v_1)=u_2$. 
In the case $\pi(u_1)=u_2$, since $G_2$ is connected, 
$\pi(v_1)$ is connected to $\pi(u_1)$ by a path of odd length, 
and thus, $\pi(v_1) \in B$. 
Any vertex $u \in B-\{v_2\}$ has degree $b$ in $G_2$, 
and $v_1$ has degree $b-1$ in $G_1$, hence, $\pi(v_1)=v_2$. 
Therefore, $\pi$ is an automorphism of $G$ 
which maps $e_1$ into $e_2$. 
The same argument can be applied to any pair of edges of $G$.
\end{proof}

\vspace{3mm}
\noindent
{\bf Acknowledgments.}
I would like to thank Joonkyung Lee for helpful discussions, 
and the three anonymous referees 
for their careful reading of the article and valuable suggestions.

\end{document}